\documentclass[11pt]{article}

\usepackage{amsmath, amssymb, amsthm}
\usepackage{hyperref}
\usepackage{dsfont}
\usepackage{enumitem}

\title{An extension of Stein's method incorporating independence}

\author{
  Aleksandar~Bala\v{s}ev-Samarski\thanks{Department of Mathematics and Computer Science, Faculty of Science, University of Sarajevo, Bosnia and Herzegovina. Email: aleksandar.bs@pmf.unsa.ba}
  \and
  Abdol-Reza~Mansouri\thanks{Department of Mathematics \& Statistics, Queen's University, Canada. Email: mansouri@queensu.ca}
}

\date{30.04.2026.} 

\newcommand{\E}[1]{\mathbb{E} \left[ #1 \right]}
\newcommand{\Esmall}[1]{\mathbb{E}[ #1 ]}
\newcommand{\N}{\mathbb{N}}
\newcommand{\R}{\mathbb{R}}

\newcommand{\norm}[1]{\left\lVert#1\right\rVert}
\newcommand{\condexp}[2]{\mathbb{E}\left[ #1 \mid #2 \right]}
\newcommand{\condexpsmall}[2]{\mathbb{E} [ #1 \mid #2 ]}
\newcommand{\ouriff}{\iff}
\newcommand{\ind}{\mathds{1}}
\newcommand{\ourimplies}{\implies}

\newtheorem{theorem}{Theorem}[section]   
\newtheorem{proposition}[theorem]{Proposition}
\newtheorem{corollary}[theorem]{Corollary}
\newtheorem{assumption}[theorem]{Assumption}

\theoremstyle{remark}
\newtheorem{remark}[theorem]{Remark}

\newcommand{\keywords}[1]{\par\noindent
  \textbf{Keywords: } #1}

\newcommand{\msc}[1]{\par\noindent
  \textbf{MSC 2020: } #1}

\begin{document}

\maketitle

\begin{abstract}
We extend Stein's method to include independence with respect to an auxiliary random variable, for any law for which a Stein characterization does exist. This extends the current literature on the problem.
Using tools from the Malliavin calculus, an application to the law of the invariant measure of an ergodic diffusion is given to illustrate the theory.
\end{abstract}


\msc{60B10, 60H07}

\section{Introduction}

Stein's method consists in characterizing a target distribution via a suitable operator
acting on a suitable class of test functions, and in using this characterization in order to derive quantitative bounds on the distance between the law of a given random variable and the target law \cite{S1972, S1986, NP2012}.
\par
More precisely, let $\mu$ be a probability measure on $\R$, which we consider as the target distribution on $\R$.
Stein's method consists in characterizing the target distribution $\mu$ by constructing a suitable operator $\mathcal{N}$ (the ``Stein operator'' for the target distribution $\mu$) acting on a suitable class of functions $\mathcal{C}$ such that 
$ \Esmall{\mathcal{N} \! f(X)} = 0, \ \forall f \in \mathcal{C}, $
if and only if the law of $X$ is equal to $\mu$. 
\par
The characterization provided by Stein's operator $\mathcal{N}$ can be used to provide bounds on the distance between the law of a given random variable $X$ and the target law $\mu$, via Stein's equation
$ \mathcal{N} \! f(x) = h(x) - \Esmall{h(M)},$
where $h$ is assumed to belong to a rich enough class
of Borel measurable functions, and $M$ is a random variable on $(\Omega, \mathcal{F}, \mathbb{P})$ having
law $\mu$. The basic intuition here is that if the law of the random variable $X$ is close to that of $M$, i.e., to $\mu$, then the expectation $\Esmall{\mathcal{N}\!f(X)}$ should be close to $0$, for any solution $f$ of the Stein equation. This general idea can be applied to Malliavin-differentiable random variables as well, with a characterization of normality in terms of Malliavin derivatives, leading to the extension known as the Malliavin-Stein method~\cite{NP2012}.
\par
In very recent work \cite{P2022, T2025} Stein's method has been extended to yield a characterization of normal random variables jointly with independence on an auxiliary random variable $Y$. 
More precisely, if $M$ is normally distributed and $Y$ is an auxiliary random variable, then, for a given random variable $X$, a Stein-type characterization is provided for the joint distribution
$\mathbb{P_{X,Y}}$ to equal the product distribution $\mathbb{P}_M \otimes \mathbb{P}_Y$. This characterization was first introduced for normally distributed random variables in~\cite{P2022}, then extended to normally distributed random vectors, with further extensions to joint asymptotic normality and independence, as well as Malliavin-differentiable random variables~\cite{T2025}, and finally to Gamma-distributed random variables in~\cite{T2024}. 
\par
It is important to note that the characterizations derived in~\cite{P2022,T2025, T2024}, as they are based on the use of characteristic functions, can a priori only cover the cases of normal and gamma distributions, owing to the very specific properties of these two distributions. Indeed, the characteristic functions for both normal and Gamma distributions satisfy a linear ordinary differential equation in which the associated differential operator is closely related to the corresponding Stein operator. This is not the case for arbitrary laws, leaving wide open the problem of characterizing independence for laws other than the normal and gamma laws, such as invariant laws of arbitrary ergodic diffusions. 
\par
In this work, we solve this problem by extending Stein's method - and Malliavin-Stein's as well - to include independence with respect to an auxiliary random variable, \emph{for any law} for which a Stein characterization does exist. We do so by circumventing the use of characteristic functions altogether. To the best of our knowledge, this is the first such extension of Stein's method.
\par
The paper is organized as follows: In Section~\ref{sec:Stein}, we recall the key ideas of Stein's method and provide the main results on which our extension is based, as well as our proposed extension. Finally, in Section~\ref{sec:Applications}, we demonstrate this extension for the case where the desired probability measure is the invariant measure of a particular ergodic diffusion. 

\section{Characterization using Stein operators and bounds on distances between laws}
\label{sec:Stein}

Consider the probability space $(\Omega, \mathcal{F}, \mathbb{P})$, and let $\mu$ be a probability measure on $\R$, which we consider as the target distribution on $\R$.
Stein's method consists in characterizing the target distribution $\mu$ by constructing a suitable operator $\mathcal{N}$ (the "Stein operator" for the target distribution $\mu$) 
satisfying the following assumption:
\begin{assumption}\label{as:1}
    \par
    There is a suitable class of functions $\mathcal{C}$ (depending on $\mathcal{N}$) such that:
    \begin{enumerate}[label=\roman*)]
        \item\label{as:i1} For any $\R$-valued random variable $X$ on $(\Omega, \mathcal{F}, \mathbb{P})$  we have the following:
        $ X \sim \mu \ouriff \Esmall{\mathcal{N} \! f(X)} = 0, \ \forall f \in \mathcal{C}; $
        \item For every bounded and measurable $h \colon \R \to \R$, the equation $\mathcal{N} \! f = h$ has a solution in $\mathcal{C}$.
    \end{enumerate}
\end{assumption}

We refer the interested reader to \cite{NP2012} for a detailed treatment of the classical Stein method.

As stated earlier, we wish to derive a characterization for a random variable $X$ to have law $\mu$ \emph{and} to be independent of an auxiliary random variable $Y$. For notational simplicity, we shall assume $Y$ is $\R-$valued; however, our results can be trivially extended to the case where $Y$ is $\R^d-$valued with $d \in \mathbb{N}$ arbitrary. The key result on which our characterization is based is given by the following proposition:

\begin{proposition}\label{pr:1}
    Let $X$ be a random variable on $(\Omega, \mathcal{F}, \mathbb{P})$, and let $\mathcal{G}$ be a sub-sigma-algebra of $\mathcal{F}$.
    We have
    \[ X \sim \mu \ \textrm{and} \ X, \mathcal{G} \ \textrm{independent} 
    \ouriff 
    \condexpsmall{\mathcal{N} \! f(X)}{\mathcal{G}} = 0 \ a.s., \ \forall f \in \mathcal{C}; \]
\end{proposition}

\begin{proof}[Proof of Proposition~\ref{pr:1}]
    Suppose first that $X \sim \mu$ and let $X$ be independent of $\mathcal{G}$.
    As $X \sim \mu$, $\Esmall{\mathcal{N} \! f(X)} = 0$ for all $f \in \mathcal{C}$.
    Furthermore, for every $B \in \mathcal{G}$ we have
    \begin{equation}\label{eq:th1}
        \Esmall{\mathcal{N} \! f(X) \ind_B} = \Esmall{\mathcal{N} \! f(X)} \Esmall{\ind_B} = 0,
    \end{equation}
    by independence of $\mathcal{G}$ and $X$.
    Since \eqref{eq:th1} holds for all $B \in \mathcal{G}$, it follows that $\condexpsmall{\mathcal{N} \! f(X)}{\mathcal{G}} = 0$ a.s.
    \par
    We now prove the converse.
    Suppose then that $\condexp{\mathcal{N} \! f(X)}{\mathcal{G}} = 0$ a.s. for all $f \in \mathcal{C}$.
    Upon taking total expectations, we obtain 
    $ \Esmall{\mathcal{N} \! f(X)} = 
       \Esmall{\condexpsmall{\mathcal{N} \! f(X)}{\mathcal{G}}} = 0, \ \forall f \in \mathcal{C},
    $
    which, by our assumptions on the operator $\mathcal{N}$ implies that $X \sim \mu$.
    We now prove independence of $X$ and $\mathcal{G}$.
    Let now $h : \R \to \R$ be Borel-measurable and bounded.
    Let $f_h \in \mathcal{C}$ be a solution to the equation 
    $ \mathcal{N} \! f = h - \E{h(X)};
    $
    note that the right-hand side of this equation is Borel-measurable and bounded, and hence such an $f_h$ does exist by assumption.
    Let now  $B \in \mathcal{G}$ be arbitrary.
    As $X \sim \mu$, we have 
    $
    \Esmall{ (h(X) - \E{h(X)})\ind_B } = \Esmall{\mathcal{N} \! f_h(X) \ind_B} 
    = \Esmall{ \condexpsmall{\mathcal{N} \! f_h(X) \ind_B}{\mathcal{G}}} 
    = \Esmall{ \ind_B \condexpsmall{\mathcal{N} \! f_h(X)}{\mathcal{G}} }
    = 0,
    $
    which then implies
    $ \Esmall{h(X) \ind_B} = \Esmall{h(X)} \Esmall{\ind_B}$. 
    Let now $A$ be an arbitrary Borel subset of $\R$, and let $h = \ind_A$.
    We then obtain $\E{\ind_A(X) \ind_B} = \E{\ind_A(X)} \E{\ind_B}$, that is, equivalently, $\mathbb{P}(X^{-1} (A) \cap B) = \mathbb{P}(X^{-1}(A)) \, \mathbb{P}(B)$.
    As $A$ and $B$ were arbitrary, the result follows.
\end{proof}

\begin{corollary}\label{co:3}
    Let $X$ be a random variable on $(\Omega, \mathcal{F},\mathbb{P})$.
    Let $(\mathcal{G}_n)_{n \in \mathbb{N}}$ be a filtration on the same probability space, let $\mathcal{G}= \bigvee_{n \in \mathbb{N}} \mathcal{G}_n$ be the sigma-algebra generated by $(\mathcal{G}_n)_{n \in \mathbb{N}}$.
    We have
    \[
    X \sim \mu \ \textrm{and} \ X, \mathcal{G}_n \ \textrm{independent for every $n$}
    \ouriff 
    X \sim \mu \ \textrm{and} \ X, \mathcal{G} \ \textrm{independent}.
    \]
\end{corollary}
\begin{proof}
Sufficiency is clear as $\mathcal{G}_n$ is a sub-sigma-algebra of $\mathcal{G}$ for every $n \in \N$; we prove necessity. Assume then 
$X \sim \mu$ and $X, \mathcal{G}_n$ independent for every $n \in \N$. We then have 
$ \condexpsmall{\mathcal{N} \! f(X)}{\mathcal{G}_n} = 0 \ a.s., \ \forall f \in \mathcal{C}, \ \forall n \in \N,$
and as
$ \condexpsmall{\mathcal{N} \! f(X)}{\mathcal{G}_n} \to \condexpsmall{\mathcal{N} \! f(X)}{\mathcal{G}} \ a.s. \ \textrm{and \ in} \ L^1$
by martingale convergence, the result follows.
\end{proof}

In what follows, $\tilde{\mathcal{C}}$ will denote the class of functions $f \colon \R \times \R \to \R$, $(x,y) \mapsto f(x,y)$, such that $f(\cdot,y) \in \mathcal{C}$ for all $y \in \R$.
In the expression $\mathcal{N} \! f(x,y)$, $\mathcal{N}$ is assumed to act only on the first argument.

\begin{corollary}\label{co:4}
    Let $X$ be a random variable on $(\Omega,\mathcal{F},\mathbb{P})$.
    We have
    \[
        X \sim \mu \ \textrm{and} \ X, \mathcal{G} \ \textrm{independent} 
        \ouriff 
        \condexpsmall{\mathcal{N} \! f(X,y)}{\mathcal{G}} = 0 \ a.s., \ \forall f \in \tilde{\mathcal{C}}, \forall y \in \R.
    \]
\end{corollary}
\begin{proof}
The sufficient direction follows directly from Proposition~\ref{pr:1}, by taking functions in $\tilde{\mathcal{C}}$ which depend only on the first argument.
To prove the converse direction, let $y \in \R$, and let $f \in \tilde{\mathcal{C}}$.
Let $f_y \colon \R \to \R$ be defined by $f_y(x) = f(x,y)$ for all $x \in \R$.
By definition of $\tilde{\mathcal{C}}$, $f_y \in \mathcal{C}$.
It follows therefore from Proposition~\ref{pr:1} that 
$    X \sim \mu \ \textrm{and} \ X, \mathcal{G} \ \textrm{independent} \ourimplies 
    \condexpsmall{\mathcal{N} \! f_y(X)}{\mathcal{G}} = 0 \ a.s.
$
which is the desired result.
\end{proof}

We define the class $\tilde{\mathcal{H}}$ to be the class of all functions $h \colon \R \times \R \to \R$, such that for all $y \in \R$, the function $h_y \colon \R \to \R$ defined by $h_y(x) = h(x,y)$ for all $x \in \R$, belongs to $\mathcal{H}$.
Let $y \in \R$, and consider the equation
\begin{equation}\label{eq:MSE}
    \mathcal{N} \! f_y(x) = h_y(x) - \E{h(M,Y)}, \ x \in \R,
\end{equation}
where $h \in \tilde{\mathcal{H}}$.
We assume the class $\tilde{\mathcal{H}}$ to be \emph{separating} on $\R^2$.
Note that for all $y \in \R$, the Stein equation \eqref{eq:MSE} has a unique solution $f_{y,h} \in \mathcal{C}$.
Defining 
$
    f_h \colon \R \times \R \to \R, \ (x,y) \mapsto f_h(x,y) = f_{y,h}(x), 
$
we can write 
\begin{equation}\label{eq:PSE}
    \mathcal{N} \! f_{h}(x,y) = h(x,y) - \E{h(M,Y)}, \ x \in \R, y \in \R.
\end{equation}
We define an appropriate class $\tilde{\mathcal{E}}$ of functions that contain all the solutions of \eqref{eq:PSE} for $h \in \tilde{\mathcal{H}}$ (and possibly more). We can now state our main result:

\begin{theorem}\label{th:1}
Let $X,Y,M$ be given random variables, with $M \sim \mu$, let $\mathcal{N}$ be the Stein operator associated with $\mu$, let $\tilde{\mathcal{H}}$ be a suitable class of test functions, and let $\tilde{\mathcal{E}}$ the class of functions containing the solutions of \eqref{eq:PSE} for $h \in \tilde{\mathcal{H}}$. We assume
$\|\mathcal{N} \! f\|_{L^\infty(\R^2)} < \infty$, for all
$f \in \tilde{\mathcal{E}}$. 
Then:
\begin{equation*} 
X \sim \mu \ \textrm{and} \ X,Y \ \textrm{independent} \ \ouriff
\ \Esmall{\mathcal{N} \! f(X,Y)} = 0, \forall f \in \tilde{\mathcal{E}}.
\end{equation*}
\end{theorem}

\begin{remark}
    The characterization obtained here is identical to those obtained in~\cite{P2022, T2025} for the normal distribution, and in~\cite{T2024} for the Gamma distribution. We point out, however, that this characterization is valid for any law for which a Stein operator does exist, and is not restricted to the normal or Gamma distributions.
\end{remark}

\begin{remark}
The condition that $\|\mathcal{N} \! f\|_{L^\infty(\R^2)} < \infty$ in the theorem is analogous to the condition $\|\partial_{x_1} f\|_{L^\infty(\R^2)}$ in~\cite{P2022}. Analogous restrictions on the function $f$ can be found in~\cite{T2024,T2025}.
\end{remark}

\begin{proof}
Assume first $\Esmall{\mathcal{N} \! f(X,Y)} = 0$ for all
$f \in \tilde{\mathcal{E}}$. Let $f_1 \in \mathcal{C}$, and let $k\colon \R \to \R$ Borel-measurable and bounded.
    We have 
    $
    \Esmall{ k(Y) \condexpsmall{\mathcal{N} \! f_1(X)}{\sigma(Y)} } = \Esmall{\mathcal{N} \! f_1(X) k(Y)} 
    = \Esmall{\mathcal{N} \! (f_1 \otimes k) (X,Y)} 
    = 0,
    $
    since $f_1 \otimes k \in \tilde{\mathcal{E}}$. Since $k$ was arbitrary, this implies
$\condexpsmall{\mathcal{N} \! f_1(X)}{\sigma(Y)} = 0 \ a.s., \ \forall f_1 \in \mathcal{E}$,
    and by Proposition~\ref{pr:1}, it follows that $X \sim \mu$ and $X,Y$ are independent.
We now prove the converse; we first assume $Y$ has finite range, with
   $Y(\Omega) = \lbrace y_1, y_2, \ldots y_N \rbrace$.
    We write $Y = \sum_{i} y_i \ind_{\lbrace Y = y_i \rbrace}$; with no loss of generality, we assume $\mathbb{P}(Y = y_i) > 0$ for all $i \in \mathbb{N}$.
It follows from Corollary \ref{co:4} that 
    $ 
        \condexpsmall{\mathcal{N} \! f(X,y)}{\sigma(Y)} = 0 \ \textrm{a.s.}, \ \forall y \in \R, \forall f \in \tilde{\mathcal{E}}.
    $
    Note now that $\condexp{\mathcal{N} \! f(X,y)}{\sigma(Y)} = \sum_i c_i(y) \ind_{\lbrace Y = y_i \rbrace}
    $, where the $c_i$ are functions depending on $y$ (as well as $f$).
    It follows that 
        $c_i(y) = 0, \ \forall y \in \R, \ \forall i \in \mathbb{N}$.
    On the other hand, we also have
    $
        \condexpsmall{\mathcal{N} \! f(X,Y)}{\sigma(Y)} = \sum_i d_i \ind_{\lbrace Y = y_i \rbrace},
    $
    where $d_i \in \R$, for all $i \in \mathbb{N}$.
    As $\condexpsmall{\mathcal{N} \! f(X,Y)}{Y \! = \! y_k} = 
        \condexpsmall{\mathcal{N} \! f(X,y_k)}{Y \! = \! y_k} = 
        c_k(y_k)$,
    it follows that $d_k = c_k(y_k) = 0$, and hence $d_k = 0$ for all $k \in \mathbb{N}$.
    Hence $\condexpsmall{\mathcal{N} \! f(X,Y)}{\sigma(Y)} = 0 \ a.s. \ \forall f \in \tilde{\mathcal{E}}$,
    and hence $\Esmall{\mathcal{N} \! f(X,Y)} = 0 \ \forall f \in \tilde{\mathcal{E}}$.
    Let now $(Y_n)_n$ be a sequence of $\sigma(Y)-$measurable simple functions converging almost surely to $Y$; since, by
    assumption, $\|\mathcal{N} \! f\|_{L^\infty(\R^2)} < \infty$, for all
$f \in \tilde{\mathcal{E}}$, it follows that     
    the sequence $(\mathcal{N} \! f(X,Y_n))_n$ is uniformly integrable $ \forall f \in \tilde{\mathcal{E}}$. Now, since for all
    $f \in \tilde{\mathcal{E}}$ we have
    $\Esmall{\mathcal{N} \! f(X,Y_n)} = 0, \ \forall n \in \N$,
    and since 
    $\mathcal{N} \! f(X,Y_n) \to \mathcal{N} \! f(X,Y)$
    almost surely, and hence in $L^1$, as $n \to \infty$, the result follows.
\end{proof}


\subsection{Bounds on the total variation and Wasserstein distances between laws using Stein equations}\label{subs:32}

Choosing an appropriate class $\tilde{\mathcal{E}}$ of functions that contain all the solutions of \eqref{eq:PSE} for $h \in \tilde{\mathcal{H}}$ (and possibly more), we obtain
\begin{equation}\label{eq:PSE1}
    \sup_{h \in \tilde{\mathcal{H}}} \left| \Esmall{h(X,Y)} - \Esmall{h(M,Y)} \right| \leq 
    \sup_{f \in \tilde{\mathcal{E}}} \left| \Esmall{\mathcal{N} \! f(X,Y)} \right|.
\end{equation}

Let then $\tilde{\mathcal{H}}_{TV}$ denote the class of all
indicator functions of Borel subsets of $\R^2$, and let
$\tilde{\mathcal{E}}_{TV}$ be a class of functions that contain all the solutions of \eqref{eq:PSE} for $h \in \tilde{\mathcal{H}}_{TV}$; similarly, let $\tilde{\mathcal{H}}_{W}$ denote the class of all Lipschitz $1$ functions on $\R^2$, and let
$\tilde{\mathcal{E}}_{W}$ be a class of functions that contain all the solutions of \eqref{eq:PSE} for $h \in \tilde{\mathcal{H}}_{W}$; We then obtain immediately:
\begin{theorem}\label{th:TVandWBound}
Let $X,Y,M$ be given random variables, with $M \sim \mu$,  and let $\mathcal{N}$ be the Stein operator associated with $\mu$. 
With $d_{\textrm{TV}}$ denoting the total variation distance, and
$d_{\textrm{W}}$ the Wasserstein distance between laws, we have:
\begin{equation*} 
    d_{\textrm{TV}} (\mathbb{P}_{X,Y}, \mathbb{P}_M \otimes \mathbb{P}_Y) \leq
    \sup_{f \in \tilde{\mathcal{E}}_{TV}} \E{\mathcal{N} \! f(X,Y)}
\end{equation*} 
and
\begin{equation*}
    d_{\textrm{W}} (\mathbb{P}_{X,Y}, \mathbb{P}_M \otimes \mathbb{P}_Y) \leq
    \sup_{f \in \tilde{\mathcal{E}}_{W}} \E{\mathcal{N} \! f(X,Y)}.
\end{equation*}
\end{theorem}

\section{Application: Invariant measure of an ergodic diffusion}
\label{sec:Applications}

In this section, we illustrate our proposed extension on the example of the invariant law of an ergodic diffusion, as the Stein operators are readily obtained in this case (\cite{BSS2005}). This is not a mere theoretical curiosity either, and knowing when a given random variable has the law of a given diffusion while being independent with respect to another random variable, is a problem of practical importance (\cite{T2024}).
\par
Consider then the diffusion given by
\[\mathrm{d} X_t = - \left( X_t-\frac{\alpha}{\alpha + \beta} \right) \textrm{d}t + \sqrt{\frac{2}{\alpha+\beta}X_t(1-X_t)} \, \textrm{d} B_t, \ t \geq 0,\]
where $\alpha,\beta > 0,$ and where $(B_t)_{t \geq 0}$ is a standard Brownian motion on $\R$. It is well known that this diffusion is ergodic, with invariant measure the $\mathrm{Beta}(\alpha,\beta)$ distribution (\cite{BSS2005}), which density is given by
\[
    p_{\alpha,\beta}(x) = \frac{\Gamma(\alpha+\beta)}{\Gamma(\alpha)\Gamma(\beta)}x^{\alpha-1}(1-x)^{\beta-1}\ind_{]0,1[}(x), \quad x \in \R.
\]
The corresponding Stein operator is given by (\cite{KT2012}):
\[
    \mathcal{N} \! f(x) = \frac{1}{\alpha + \beta} \,x (1-x) f'(x) - \left( x - \frac{\alpha}{\alpha + \beta} \right) f(x), \ x \in [0,1],
\]
and for all $y \in \R$, the unique solution to the Stein equation 
\[\mathcal{N} \! f(x,y) = h(x,y) - \E{h(M,y)}, \ x \in [0,1],\] 
with the operator $\mathcal{N}$ acting only on the first variable, is given by
\begin{equation}\label{eq:BetaSteinSol}
    f_h(x,y) = (\alpha + \beta) x^{-\alpha}(1-x)^{-\beta} \int_0^x 
    \frac{h(u,y) - \E{h(M,y)}}{u^{1-\alpha} (1-u)^{1-\beta}} \, \mathrm{d}u, \ x \in [0,1].
\end{equation}
For $h$ $1-$Lipschitz, there exist constants $C_1,C_2,C_3 > 0$ depending only on $\alpha,\beta$ such that $\norm{\frac{\partial f_h}{\partial x}}_\infty \leq C_1$, $\norm{\frac{\partial f_h}{\partial y}}_\infty \leq C_2$, and  $\norm{f_h}_{\infty} \leq C_3$ (\cite{D2012}).
\par
We now specialize to the case $\alpha=\frac{1}{2},\beta=1,$ in which case the diffusion becomes
\[ \mathrm{d} X_t = -\left( X_t-\tfrac{1}{3} \right) \mathrm{d} t + \sqrt{\tfrac{4}{3}X_t(1-X_t)} \, \mathrm{d} B_t, \ t \geq 0,\]
and the corresponding Stein operator is given by
\begin{equation}\label{eq:soe}
    \mathcal{N} \! f(x) = \tfrac{2}{3} \,x (1-x) f'(x) - \left( x - \tfrac{1}{3} \right) f(x), \ x \in [0,1].
\end{equation}
We assume given an isonormal Gaussian process $W=\{W(\mathfrak{h}): \mathfrak{h} \in \mathfrak{H}\}$ defined on a suitable probability space $(\Omega,\mathcal{F},\mathbb{P}),$ where $\mathfrak{H}$ is a real separable Hilbert space with inner product $\langle \cdot,\cdot \rangle_{\mathfrak{H}}$, as well as random variables
$X,Y \in \mathbb{D}^{1,2}$, with $\mathbb{D}^{1,2}$ denoting the domain in $L^2(\Omega)$ of the Malliavin derivative operator $D$; We denote by $\delta$ the adjoint of $D$, and by $L=-\delta D$ the infinitesimal generator of the corresponding Ornstein-Uhlenbeck semigroup (\cite{N2006}). With the random variable $M$
$\mathrm{Beta}(\frac{1}{2},1)-$distributed, we wish to compute a bound on the Wasserstein distance $d_{\textrm{W}} (\mathbb{P}_{X,Y}, \mathbb{P}_M \otimes \mathbb{P}_Y)$. In order to apply Theorem~\ref{th:TVandWBound}, we need to estimate 
$\vert \Esmall{\mathcal{N} \! f_h(X,Y)} \vert$ 
for $h$ $1-$Lipschitz, where $f_h$ is the solution to the corresponding Stein equation and is given by \eqref{eq:BetaSteinSol} with
$\alpha=\frac{1}{2}$ and $\beta=1$.
\par
With the Stein operator $\mathcal{N}$ of \eqref{eq:soe} acting only on the first variable, we have, upon taking expectations and after a slight rearrangement:
\begin{align*}
   \Esmall{\mathcal{N} \! f_h(X,Y)} = 
   \E{\tfrac{2}{3} X(1-X)\partial_xf_h(X,Y) -
   (X - \E{X})f_h(X,Y)} \\
   - \E{ \left( \E{X}-\tfrac{1}{3} \right) f_h(X,Y)}.
\end{align*}
Applying the Malliavin integration by parts formula, we obtain
\begin{multline*}
 \Esmall{(X - \Esmall{X}) f_h(X,Y)} = \Esmall{(\delta D) (-L)^{-1}(X - \Esmall{X}) f_h(X,Y)} \\
    = \Esmall{\partial_x f_h(X,Y) \langle DX,-DL^{-1}(X - \Esmall{X}) \rangle_{\mathfrak{H}} } 
    \\ + \Esmall{\partial_y f_h(X,Y) \langle DY,-DL^{-1}(X - \Esmall{X}) \rangle_{\mathfrak{H}}},
\end{multline*}
which yields
\begin{multline*}
    \Esmall{\mathcal{N} \! f_h(X,Y)} = 
        \E{\partial_x f_h(X,Y) \left( \tfrac{2}{3} X(1-X) - \langle DX,-DL^{-1}(X -\Esmall{X})\rangle_{\mathfrak{H}} \right)} \\
    + \Esmall{\partial_y f_h(X,Y) \langle DY,DL^{-1}(X - \Esmall{X}) \rangle_{\mathfrak{H}}}
    - \E{ \left( \Esmall{X}-\tfrac{1}{3} \right) f_h(X,Y)},
\end{multline*}
and hence,
\begin{multline}\label{eq:nibp}
     d_{\textrm{W}} (\mathbb{P}_{X,Y}, \mathbb{P}_M \otimes \mathbb{P}_Y) \leq 
        C_1 \E{ \left| \tfrac{2}{3} X(1-X) - \langle DX,-DL^{-1}(X - \Esmall{X})\rangle_{\mathfrak{H}} \right| } \\
    + C_2 \E{ \left| \langle DL^{-1}(X - \Esmall{X}), DY \rangle_{\mathfrak{H}} \right| }
    + C_3 \E{ \left| \Esmall{X}-\tfrac{1}{3} \right| }.
\end{multline}
Let now the sequences $(X_n)_{n \in \mathbb{N}}$ and 
$(Y_n)_{n \in \mathbb{N}}$ in $\mathbb{D}^{1,2}$ be given by
\[
X_n = e^{-I_1(\mathfrak{h}_1)^2 - \frac{n}{n+1} I_2(\mathfrak{h}_2)^2}, \ 
Y_n = \frac{1}{n+1} I_1(\mathfrak{h}_1) + I_1(\mathfrak{h}_3), \quad n \in \mathbb{N},
\]
where $(\mathfrak{h}_i)_{i \in \mathbb{N}}$ is a Hilbert basis for $\mathfrak{H}$. As a simple application of the bound on Wasserstein distance obtained in Equation \eqref{eq:nibp}, we prove that $d_{\textrm{W}} (\mathbb{P}_{X_n,Y_n}, \mathbb{P}_M \otimes \mathbb{P}_{Y_n}) \to 0$
as $n \to \infty$, that is, $X_n$ is asymptotically $\mathrm{Beta}(\frac{1}{2},1)$ distributed \emph{and} $X_n,Y_n$ are asymptotically independent; a finer analysis would also yield the rate of convergence to $0$.
\par
Note first that $X = e^{-I_1(\mathfrak{h}_1)^2 - I_2(\mathfrak{h}_2)^2}$ is
$\mathrm{Beta}(\frac{1}{2},1)-$distributed, and by dominated convergence, 
$ \Esmall{ | \Esmall{X_n}-\tfrac{1}{3} | } \to 0 $
as $n \to \infty$.
\par
We now deal with the first term of \eqref{eq:nibp}, and 
to do so, we use the following result from \cite{NV2009}: Letting $Z = h(N)$, where $h: \mathbb{R}^n \to \mathbb{R}$ is of class $C^1$ 
with bounded derivatives, $N = (N_1, \ldots, N_n)$ is a Gaussian vector with zero mean and covariance matrix $K = (K_{i,j})_{i,j=1,\ldots,n}$, and $N'$ is an independent copy of $N$ (with $N,N'$ defined on a product probability space $(\Omega \times \Omega', \mathcal{F} \otimes \mathcal{F}, \mathbb{P} \times \mathbb{P}')$), and with $\mathbb{E}'$ denoting the expectation with respect to the probability measure $\mathbb{P}^\prime$,
the following formula holds: 
\begin{multline*}
    \langle D(-L)^{-1}(Z - \Esmall{Z}),DZ \rangle_{\mathfrak{H}} 
    \\ = 
        \int_0^\infty e^{-u} \, \mathbb{E}'\left[  \sum_{i,j=1}^n K_{i,j} \partial_{x_i}h(N) \partial_{x_j} h\left( e^{-u}N + \sqrt{1-e^{-2u}} N' \right) \right] \mathrm{d}u.
\end{multline*}
In our case, $n = 2$, $N_1 = I_1(\mathfrak{h}_1), N_2 = I_2(\mathfrak{h}_2)$, $K_{i,j} = \delta_{ij}$.
After the change of variables $e^{-u} = a$, repeated applications of Lemma 1 from \cite{KT2012} yield, after somewhat involved but straightforward calculations:
\begin{multline*}
\langle D(-L)^{-1}(X_n - \Esmall{X_n}),DX_n \rangle_{\mathfrak{H}} = 4e^{-(N_1^2 + \frac{n}{n+1}N_2^2)} \times \\
    \int_0^1 \frac{a \left( \left( 1 + 2 \frac{n}{n+1}(1-a^2) \right)  N_1^2 + 
        \frac{n^2}{(n+1)^2}(3-2a^2)N_2^2 \right)}
        {(3 - 2a^2)^{3/2}(1 + 2\frac{n}{n+1}(1-a^2))^{3/2}}
    e^{-a^2 \left(\frac{N_1^2}{3 - 2a^2} + \frac{N_2^2}{\frac{1}{n} + 3 - 2a^2} \right)} \mathrm{d} a.
\end{multline*}
Since $X$ is $\mathrm{Beta}(\frac{1}{2},1)-$distributed, and hence satisfies the Malliavin-Stein equation
\[ \langle DX,-DL^{-1}(X-\Esmall{X})\rangle_{\mathfrak{H}} = \tfrac{2}{3} X(1-X) \ a.s., \]
and since $X_n \to X$ pointwise on $\Omega$, it follows immediately from dominated convergence that 
\[ \E{ \left| \tfrac{2}{3} X_n(1-X_n) - \langle DX_n,-DL^{-1}(X_n-\Esmall{X_n})\rangle_{\mathfrak{H}} \right| } \to 0\ \, \textrm{as $n \to \infty$.} \]
\par
We now show that $ \Esmall{|\langle DL^{-1}(X_n-\E{X_n}), DY_n \rangle_{\mathfrak{H}} |} \to 0$
as $n \to \infty$. It follows from Proposition 1.4.4 of \cite{N2006} that
$ DL^{-1}(X_n - \Esmall{X_n}) \in L^2(\Omega) \otimes \mathfrak{H}_{1,2}, \ \forall n \in \mathbb{N}$,
where $\mathfrak{H}_{1,2}$ is the $\R-$linear span of $\{\mathfrak{h}_i\}_{i=1}^2$ in the Hilbert space $\mathfrak{H}$. Hence,
\[ \langle DL^{-1}(X_n-\Esmall{X_n}), DY_n \rangle_{\mathfrak{H}} = 
\frac{1}{n+1} \langle DL^{-1}(X_n-\Esmall{X_n}), \mathfrak{h}_1 \rangle_{\mathfrak{H}}, \]
and, using the commutation relation $ D L^{-1} F = -(I - L)^{-1} DF$
for $F \in \mathbb{D}^{1,2}$ with $\Esmall{F}=0$ (\cite{NP2012}), 
we obtain
\[ \langle DL^{-1}(X_n-\Esmall{X_n}), DY_n \rangle_{\mathfrak{H}} 
= -\frac{1}{n+1} \langle (I-L)^{-1}DX_n, \mathfrak{h}_1 \rangle_{\mathfrak{H}}. \]
Since $(I-L)^{-1}$ is a contraction on $L^2(\Omega)$ (\cite{N2006}), and since $(X_n)_{n \in \mathbb{N}}$ is bounded in $\mathbb{D}^{1,2}$, an application of Cauchy-Schwarz yields $C > 0$ such that
\[ \Esmall{|\langle (I-L)^{-1}DX_n, \mathfrak{h}_1 \rangle_{\mathfrak{H}}|} \leq C, \ \ \forall n \in \mathbb{N}, \]
from which the result follows.

\section*{Acknowledgments}
The authors sincerely thank the anonymous reviewer for their very careful reading of the manuscript and for their constructive
comments and suggestions, all of which have greatly helped improve the clarity and quality of this work. The authors also wish to thank Dr. Zachary Selk for valuable comments.
This work was supported in part by the Natural Sciences and Engineering Research Council (NSERC) of Canada.

\section*{Data availability}
No data was used for the research described in the article.

\bibliographystyle{plain} 
\bibliography{references} 

\end{document}